\documentclass[12pt]{amsart}
\usepackage{geometry}
\usepackage{graphicx}
\usepackage{amssymb}
\usepackage{epstopdf}
\usepackage{amsmath,amscd}
\usepackage{amsthm}
\usepackage{url,verbatim}
\usepackage{stmaryrd}
\usepackage{longtable}
\usepackage[table]{xcolor}

\RequirePackage[colorlinks,citecolor=blue,urlcolor=blue]{hyperref}
\usepackage{breakurl}
\theoremstyle{plain}
\newtheorem{theorem}{Theorem}

\newtheorem{lemma}[theorem]{Lemma}

\newtheorem{corollary}[theorem]{Corollary}

\newtheorem{remark}[theorem]{Remark}
\newtheorem{question}[theorem]{Question}

\newcommand\es{\varnothing}

\newcommand\ol{\overline}
\newcommand\Aut{\mathrm{Aut}}

\newcommand\RR{{\mathbb R}}
\newcommand\ZZ{{\mathbb Z}}
\newcommand\NN{{\mathbb N}}
\newcommand\NNp{{\mathbb N_0}}

\newcommand\LL{{\mathbb L}}

\newcommand\om{\omega}
\newcommand\g{\gamma}
\newcommand\be{\beta}
\newcommand\si{\sigma}

\newcommand\qq{\qquad}
\newcommand\q{\quad}
\newcommand\resp{respectively}

\newcommand\oo{\infty}
\newcommand\sG{{\mathcal G}}

\newcommand\sK{{\mathcal K}}

\newcommand\sS{{\mathcal S}}
\newcommand\sW{{\mathcal W}}
\newcommand\sV{{\mathcal V}}
\newcommand\Ga{\Gamma}

\newcommand\de{\delta}

\newcommand\id{{\bf 1}}

\newcommand\pc{p_{\text{\rm c}}}

\newcommand\Om{\Omega}
\newcommand\la{\lambda}
\newcommand\what{\widehat}

\newcommand\pd{\partial}

\newcommand\setm{\setminus}
\newcommand\ghf{graph height function}

\newcommand\ccl{\cellcolor{blue!25}}
\newcommand\cclr{\cellcolor{pink}}
\newcommand\larrow{\longrightarrow}

\newcommand\CG{\text{\rm CG}}

\newcounter{mycount1}\newcounter{mycount2}\newcounter{mycount3}\newcounter{mycount}

\newenvironment{numlist}{\begin{list}{\rm\arabic{mycount2}.}%
   {\usecounter{mycount2}\labelwidth=1cm\itemsep 0pt}}{\end{list}}
\newenvironment{letlist}{\begin{list}{\rm(\alph{mycount3})}%
   {\usecounter{mycount3}\labelwidth=1cm\itemsep 0pt}}{\end{list}}

\numberwithin{equation}{section}
\numberwithin{theorem}{section}
\numberwithin{figure}{section}
\numberwithin{table}{section}

\title{Connective constants of Grigorchuk graphs}
\author{Geoffrey R.\ Grimmett}
\address{Statistical Laboratory, Centre for
Mathematical Sciences, Cambridge University, Wilberforce Road,
Cambridge CB3 0WB, UK} 
\email{grg@statslab.cam.ac.uk}
\urladdr{\url{http://www.statslab.cam.ac.uk/~grg/}}

\begin{document}

\begin{abstract}
The connective constant $\mu(G)$ of a graph $G$ is the exponential growth rate
of the number of self-avoiding walks starting at a given vertex. We prove upper and lower bounds for the 
connective constants of Cayley
graphs $G_\om$ of a general Grigorchuk group  encoded by a sequence $\om\in\{0,1,2\}^\NN$.
In particular, $\mu(G_\om) > \phi$ for any such Cayley graph (subject to a simple condition on $\om$), 
where $\phi:= \frac12(1+\sqrt 5)$ is
the golden mean. 
This extends earlier work of the author and Zhongyang Li
in \lq\lq Cubic graphs and the golden mean'', Discrete Math.\ 343 (2020), article 111638,
where it was conjectured that $\mu(G)\ge\phi$ for all infinite, vertex-transitive, cubic graphs.
The current work includes an analysis of the proportions of appearances of given label-sequences in the 
orbital Schreier graphs of general Grigorchuk groups. 
\end{abstract}

\date{August 21, 2026, revised October 1, 2026}

\maketitle

\section{Introduction}\label{sec:int}

Let $G$ be an infinite, (vertex-)transitive, simple, rooted graph, 
and let $\si_n$ be the number of $n$-step self-avoiding walks (SAWs) starting from the root.
It was proved by Hammersley \cite{jmhII} in 1957 that the limit
$\mu=\mu(G):=\lim_{n\to\oo}\si_n^{1/n}$ exists, 
and he called $\mu$ the \lq connective constant' of $G$. 
A great deal of attention has been devoted to counting SAWs since that introductory mathematics paper, 
and survey accounts of many of the main features of the theory may be found at \cite{bdgs,GL-amen, ms}.

A graph is called \emph{cubic} if every vertex has degree $3$, and \emph{transitive} if it is vertex-transitive.
Let $\sG_d$ be the set of connected, infinite, transitive, simple graphs with degree $d$, and
let $\mu(G)$ denote the connective constant of $G \in \sG_d$. 
The letter $\phi$ denotes the golden mean $\phi:=\frac12(1+\sqrt 5) \approx 1.618\cdots$. The following question was stated in \cite{GL-Comb} and investigated but not fully answered in \cite{GL20}. 

\begin{question}\label{qn:big}
Is it the case that $\mu(G) \ge \phi$ for $G \in \sG_3$?
\end{question}

Our principal purpose is to answer this question affirmatively for the Cayley graphs of Grigorchuk groups,
therein extending \cite[Thm 8.1]{GL20}.
We state this loosely here, with a mathematically precise, quantified, statement provided at Theorem \ref{thm:Grig2}. We recall that the general Grigorchuk group is encoded by an 
infinite sequence $\om=(\om_0,\om_1,\om_2,\dots)\in\{0,1,2\}^\oo$; 
see Section \ref{sec:Grig} for the details of this. 
Let $Q$ be the set of all sequences $\om$ that include all three values
from the set $\{0,1,2\}$.

The Grigorchuk groups are usually defined with four generators, of which three (taken with the identity)
form a Klein $4$-group. Therefore,
a Grigorchuk group has three distinct $3$-generator Cayley graphs.

\begin{theorem}\label{thm:0}
Let $\om\in\{0,1,2\}^\oo$.
Each Cayley graph $\CG_\om$ of a ($3$-generator) general Grigorchuk group $\Ga_\om$
satisfies $\mu(\CG_\om)\ge\phi$. Furthermore, the inequality is strict if $\om\in Q$.
\end{theorem}

See Theorem \ref{thm:Grig2} for a fuller statement that includes rigorous upper 
and lower bounds for the connective constants in question.

Recall that the \lq ladder graph' of Figure \ref{fig:ladder} has connective constant $\phi$.
One useful way to show that $G\in\sG_3$ satisfies $\mu(G)\ge\phi$ is to find a injection from the set of SAWs on $\LL$ into
the set of SAWs on $G$. That is not the approach used here.

\begin{figure}
\centerline{\includegraphics[width=0.7\textwidth]{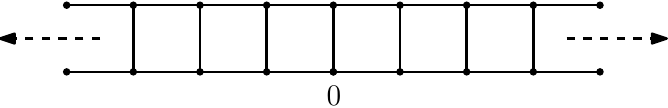}}
\caption{The ladder graph $\LL$ with marked root. }
\label{fig:ladder}
\end{figure}

Several categories of cubic graphs are shown in \cite{GL20}  to satisfy $\mu\ge\phi$, but the general question
remains open. One may consider its restriction to Cayley graphs of finitely generated groups
(all Cayley graphs in this paper are in their \emph{simple} form, that is, edges are undirected, and
parallel edges are allowed to coalesce). 

It is explained in \cite{GL20} that the SAW analysis of such Cayley graphs
varies depending on the number of ends of the group. Theorem 10.2 of \cite{GL20} proves the required  inequality
for two-ended, cubic, Cayley graphs, and we note in passing the extension of this by Georgakopoulos and Wendland
\cite{GW23} to all two-ended, transitive, cubic graphs.

\begin{theorem}  [\mbox{\cite[Thm 1.3]{GW23}, \cite[Thm 10.2]{GL20}}]  \label{thm:2endgp}
Let $d\ge 3$ and let $G\in\sG_3$ be  two-ended. 
Its connective constant
satisfies $\mu(G) \ge \phi$.  
\end{theorem}

The proof, which is found in \cite{GW23}, proceeds as in \cite{GL20} by constructing a non-constant, Lipschitz, 
\lq skew-difference-invariant' harmonic function
on $G$.
The paper \cite{GW23} contains also
an example of a two-ended, transitive, cubic graph that is not a Cayley graph, thereby answering an old question of
Watkins \cite{Wat}.

The Grigorchuk \lq branch groups' form a category of groups that has attracted great interest, and this note
is devoted to such groups. The work reported here extends \cite[Thm 8.1]{GL20} which was concerned
with the so-called first 
Grigorchuk group. In the context of self-avoiding walks, Grigorchuk groups are special in that they
do not possess so-called \ghf s (see \cite[Cor.\ 9.2]{GL-nonamen} and \cite[Thm C]{AB}).

We remind the reader of the use of generating functions in the context of SAWs. Let $\sW$ be the set of SAWs 
from a given root in an infinite, transitive  graph $G$, and let $\sigma_n$ be the number of
such $n$-step walks. The SAW generating function
is given as
\begin{equation}\label{eq:ZG}
Z_G(\xi) = \sum_{w\in\sW} \xi^{|w|} = \sum_n \sigma_n \xi^n,\qq \xi\in\RR,
\end{equation}
where $|w|$ is the number of edges of $w$. The radius of convergence of $Z_G$ is $1/\mu(G)$. 

This paper is laid out as follows. 
Grigorchuk groups are defined in Section \ref{sec:Grig}, 
the main theorem is presented in Section \ref{sec:main},
and it is proved in Section \ref{sec:pfmain}. 

One of  the contributions of the current work is a direct derivation of the asymptotic proportions of 
appearances of given edge-label patterns in the orbital Schreier graphs of general Grigorchuk groups indexed 
by a vector $\om$. This is described in Section \ref{sec:abc}, and the relevant theorem is found at Theorem \ref{thm:label}.

The set $\{0,1,2,\dots\}$ of non-negative integers is
denoted $\NNp$.  Given a finite string $w$, we write $w^k$ (\resp, $w^\oo$) for
the string $ww\cdots w$ comprising $k$ copies of $w$ (\resp, a singly infinite string $ww\cdots$
of copies of $w$).

\section{The Grigorchuk groups}\label{sec:Grig}

\subsection{The first Grigorchuk group}

The (first) Grigorchuk group $\Ga$ was introduced in \cite{Grig80} (see also the more recent
papers \cite{Grig84, RG05, GLN}) as a group of intermediate growth. It
is defined as follows.
Let $T$ be the rooted binary tree with root labelled $\es$. 
The vertex-set of $T$ can be identified with the set of finite strings $u$ having entries $0$, $1$, 
where the empty string corresponds to the root $\es$,
 and the two children of a vertex labelled $u$ are labelled $u0$ and $u1$.
Let $T_u$ denote the subtree of $T$ with root $u$ together with its \lq offspring'. 

Let $\Aut(T)$ be the automorphism group of $T$, and
let $a\in\Aut(T)$ be the automorphism
that, for each string $u$, interchanges the two vertices $0u$ and $1u$. 

Any $\g\in\Aut(T)$ may be applied
in a natural way to either subtree $T_u$, $u=0,1$. 
Given two elements 
$\g_0,\g_1\in\Aut(T)$, we define $\g=(\g_0,\g_1)$ to be the automorphism
of $T$ obtained by applying $\g_0$ to $T_0$ and $\g_1$ to $T_1$.
Define automorphisms $b$, $c$, $d$ of $T$ recursively as follows:
\begin{equation}\label{eq:grigrels}
b=(a,c),\quad c=(a,d),\quad d=(\id,b),
\end{equation}
where $\id$ is the identity automorphism. 

The first Grigorchuk group $\Ga$ is defined as the subgroup of  $\Aut(T)$ generated by the 
set $S:=\{a,b,c,d\}$. 
Since every generator has order $2$, one may define its undirected Cayley graph
$G$ with degree $4$. We re-use the notation $\id$ for the identity element of $\Ga$.

It is standard and easily checked that
\begin{align}
\label{eq:commute}
b=cd=dc,\q
c=db=bd,\q
d=bc=cb.
\end{align}
Therefore, the set $S$ is not a minimal generator set of $\Ga$. More specifically,
$\Ga$ is generated by any set of the form $S(\neg z) := S\setm\{z\}$ for any given $z\in\{b,c,d\}$.
We write $G(\neg z)$ for the (degree-$3$) Cayley graph of $\Ga$ with generator-set $S(\neg z)$.

The $3$-neighbourhood of  $\id$ in the Cayley graph $G(\neg d)$ of $\Ga$
is drawn in Figure \ref{fig:G} (subject to a relabelling of the vertices which will be clearer later in this work). 

\begin{figure}[htbp]
\centerline{\includegraphics*[width=0.45\hsize]{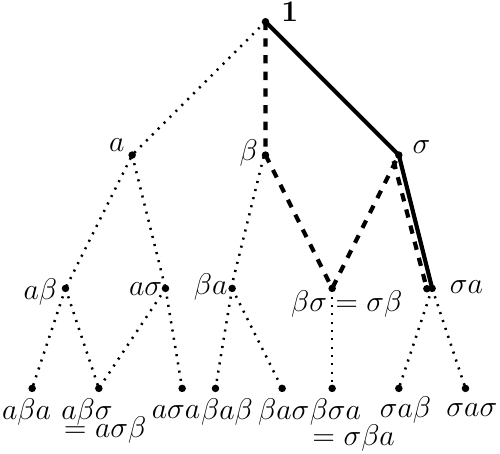}}
   \caption{The $3$-neighbourhood of the identity in $G_\om(\neg \de)$. Note that the walk $(\id,\sigma ,\sigma a)$
   may be re-routed as $(\id,\beta,\beta\sigma,\sigma,\sigma a)$.}
   \label{fig:G}
\end{figure}

\subsection{General Grigorchuk groups}

We turn now to the general Grigorchuk groups of \cite{Grig84} (see also \cite{Bar17,MP99}), which
are encoded as follows in terms of an infinite  vector $\om$ of $0$s, $1$s, and $2$s.
The first Grigorchuk group corresponds to the vector $012012012\dots$, abbreviated to $(012)^\oo$.

Let automorphisms $I(u)$ and $\Pi(u)$ act at vertices $u$ of $T$ in the following manner. The map
$I(u)$ is the identity map that makes no change to $T_u$, whereas $\Pi(u)$ interchanges
the two children $u0$ and $u1$ and their sub-trees.

\begin{figure}[htbp]
\centerline{\includegraphics*[width=0.7\hsize]{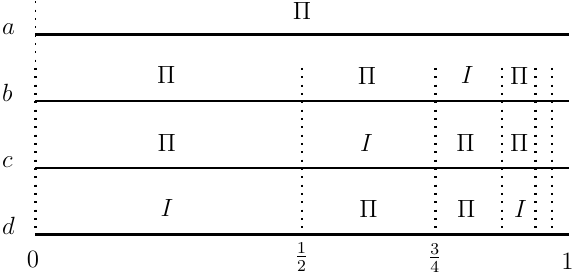}}
   \caption{The unit interval $[0,1)$ may be partitioned into the sub-intervals $[1-2^{-n},1-2^{-n-1})$ for $n\ge 0$.
   The generators $a$, $b$, $c$, $d$ may be expressed as involutions of these sub-intervals.
   Each involution is either the identity $I$, or the permutation $\Pi$ that
   interchanges the two halves of the corresponding sub-interval.  Recalling \eqref{eq:012}--\eqref{eq:wedge}, 
   the illustrated  example is encoded in the word $(012)^\oo$. 
   The more general generator-set $\{a, b_\om, c_\om,d_\om\}$
   may be described similarly by adapting the columns to the elements of the vector $\om$.}
   \label{fig:int}
\end{figure}

Define the symbolic $3$-vectors
\begin{equation}\label{eq:012}
\ol 0 = \begin{pmatrix} \Pi \\ \Pi \\ I\end{pmatrix} ,\quad
\ol 1 = \begin{pmatrix}\Pi \\ I \\ \Pi\end{pmatrix},\quad
\ol 2 = \begin{pmatrix} I \\ \Pi \\ \Pi \end{pmatrix}.
\end{equation}
Let $\om=(\om_n: n=0,1,2,\dots) \in \Om :=\{0,1,2\}^\NNp$. 
(This indexing convention differs slightly from the usual one, but is convenient in the current setting.)
Combine the column vectors $\ol {\om_0}, \ol{\om_1},\dots$ into a $3\times \oo$ matrix $M$
and let $(b_n: n=0,1,2,\dots)$ (\resp, $(c_n)$, $(d_n)$)
be the top row-vector (\resp, second, third row-vector) of $M$. 
We define four automorphisms $a$, $b_\om$, $c_\om$, $d_\om$ by $a=\Pi(\es)$, and
\begin{align}\label{eq:wedge}
b_\om &= b_0(0)\circ b_1(10) \circ \cdots \circ b_k(1^k0)\circ \cdots,
\end{align}
with a similar formula for $c_\om$ (\resp, $d_\om$) with each $b_n$ replaced by $c_n$ (\resp, $d_n$).
Here, 
$b_k(1^k0)$ is an automorphism acting as $b_n$ on the sub-tree $T_{1^k0}$.
In \eqref{eq:wedge}, for automorphisms $\alpha$, $\beta$ that act on disjoint vertex-subsets of $T$,
the notation $\alpha\circ \beta$ means that they act simultaneously. 

It is helpful to illustrate these generators
as involutions of sub-intervals of $[0,1)$, in the style of the initial papers in the field, namely \cite{Grig80,Grig84}.
See Figure \ref{fig:int}. 

The general Grigorchuk group $\Ga_\om$ is defined as the group generated by
the set $S_\om:=\{a, b_\om, c_\om, d_\om\}$, with Cayley graph denoted $G_\om$. 
Since $\Pi^2=I$,
the elements of $S_\om$  have order $2$ and,  since $\Pi^2=I$ and $\Pi I=I\Pi=I$, we have that
\begin{equation}\label{eq:gcommute}
b_\om=c_\om d_\om=d_\om c_\om,\ 
c_\om=d_\om b_\om=b_\om d_\om,\ 
d_\om=b_\om c_\om=c_\om b_\om, \qq\text{for } \om\in \Om.
\end{equation}
Therefore, $S_\om$ is not a minimal generating set of $\Ga_\om$, and
one may discard any single element of $\{b_\om, c_\om, d_\om\}$
to obtain an undirected, degree-$3$, Cayley graph denoted $G_\om(\neg z)$ for $\Ga_\om$, 
where $z$ denotes the deleted generator. 
By \eqref{eq:gcommute}, the set $\{\id,b_\om,c_\om,d_\om\}$ forms a Klein $4$-group.

The $4$-generator Cayley graph $G_\om$ of $\Ga_\om$ is sketched in Figure \ref{fig:G2}. A notable
feature of this graph is 
the regular appearance of disjoint copies of the complete graph $K_4$, each of which represents a translate of
the above Klein $4$-group. These subgraphs are termed \emph{kites}, and we write $\sK$ for the set of all kites.

\begin{figure}
\centerline{\includegraphics*[width=0.45\hsize]{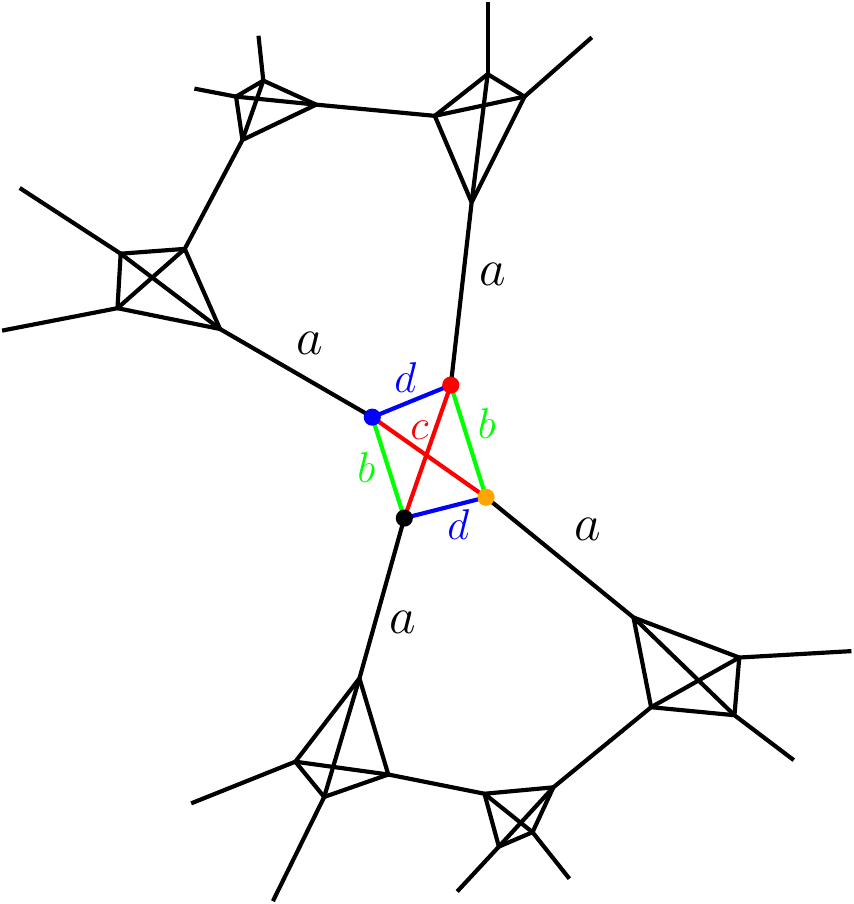}\quad\raisebox{2cm}{\includegraphics*[width=0.2\hsize]{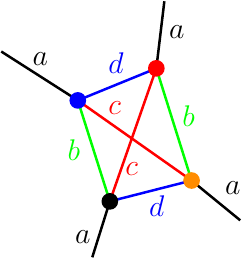}}}
   \caption{Part of the Cayley graph $G$ of the first Grigorchuk group $\Ga$. The Cayley graph $G_\om$
   of the general group $\Ga_\om$ is similar but the cycle structure may differ (though the kites remain).
   The letters $b$, $c$, $d$ are then shorthand for $b_\om$, $c_\om$, $d_\om$.
   \emph{Left}: The black disk is the identity $\id$ and the central kite is the
   complete graph on the four vertices $\id$, $b$, $c$, $d$. 
   \emph{Right}: The central kite when expanded. }
   \label{fig:G2}
\end{figure}

\section{Main theorem}\label{sec:main}

Let $\om\in\Om$.
Recall that $G_\om\in\sG_4$, and $G_\om(\neg z) \in \sG_3$ for $z\in \{b_\om,c_\om,d_\om\}$.

\begin{theorem}\label{thm:Grig2}
Let $\om\in\Om$. 
\begin{letlist}
\item
The connective constant $\mu(G_\om)$ satisfies 
$\tau\le \mu < \tau'$
where $\tau\approx 2.167$ is the reciprocal of the positive root of $2\xi^2(1+2\xi+2\xi^2)-1=0$,
and $\tau'\approx 2.508$ is the reciprocal of the positive root of the equation
\begin{equation}\label{eq:600}
B(\xi)\bigl(1+D(\xi)+D(\xi)^2+D(\xi)^3\bigr)-1=0,
\end{equation}
 where $B(\xi)=2D(\xi)$ and $D(\xi)=\xi^2(1+2\xi+2\xi^2)$.

\item
Let $\de=\de(\om)\in\{b_\om,c_\om,d_\om\}$ be such that $\de_0=I$. 
The connective constant $\mu(G_\om(\neg \de))$ satisfies 
$\gamma \le \mu  < \gamma'$,
where $\gamma\approx1.635$ is the reciprocal of the positive root of $2\xi^4(1 + \xi^2)(1 +\xi)^2-1=0$,
and $\gamma'\approx 1.895$ is the reciprocal of the positive root of \eqref{eq:600}
with $B(\xi)=2\xi^2(1+\xi^2)$ and $D(\xi)=2\xi^3$.


\item
Let $\beta$, $\sigma$ be the two elements of $\{b_\om,c_\om,d_\om\}$  other than the $\de$ of part (b), and let $z\in\{\beta,\sigma\}$.
The connective constant $\mu(G_\om(\neg z))$  satisfies $\phi \le \mu  < \gamma''$,
where $\phi$ is the golden mean and $\gamma'' \approx 1.892$ is the reciprocal of the positive root of 
\eqref{eq:600} with $B(\xi)=\xi^2(1+\xi)^2$ and $D(\xi)=\xi^2(1+\xi^2)$.
Furthermore, the strict inequality $\mu(G_\om(\neg  z)) > \phi$ holds if $\om$ is such that
there exists $k\ge 1$ with $z_k=z_k(\om)=I$. 

\end{letlist}
\end{theorem}

The proof is found in Section \ref{sec:pfmain}. The proof of part (c) makes use of the results of Section
\ref{sec:abc} concerning the densities of label-sequences in the 
orbital Schreier graph of the ray $1^\oo$ of the tree $T$. 

The bounds of Theorem \ref{thm:Grig2} provide some
information about the numerical values of the connective constants,
and they may presumably be improved by using more intricate combinatorics. 
Rudi Mrazovi\'c \cite{RM} has kindly pointed out some numerical improvements to
the bounds reported in an earlier version of this paper.

\begin{remark}\label{rem:2}
When $\om=(012)^\oo$, 
$\Ga_\om$ is the first Grigorchuk group $\Ga$, in which case 
$b_\om = b=\beta$, $c_\om=c=\sigma$, and $d_\om=d=\delta$.
A weaker version of part (b) was proved in \cite[Thm 8.1]{GL20} in this special case, 
namely that $\mu(G_\om(\neg d))\ge \phi$. The proof in Section \ref{sec:pfmain} of the lower bound of
Theorem \ref{thm:Grig2}(b) is a slightly streamlined version of that used in \cite{GL20}.
\end{remark}

\begin{remark}\label{rem:4}
Note that the (weak) numerical bounds in Theorem \ref{thm:Grig2}(b, c) depend 
primarily on the value of $\om_0$.
\end{remark}

\begin{remark}\label{rem:1}
The sets of SAWs used in the proofs of the lower bounds of Theorem \ref{thm:Grig2} are restricted to walks $w$
with two further properties: (i) $w$ does not arrive twice (or more)
in any given kite,  and (ii) $w$ may be extended to
an infinite SAW with the foregoing property. (Infinitely extendable SAWs were studied in \cite{GHP}.)
The final strict inequality of part (c) may be made explicit by exploring the positive root of 
the forthcoming equation \eqref{eq:improve}.
\end{remark} 

\begin{remark}\label{rem:3}
Muchnik and Pak \cite{MP99} considered the critical value $\pc(G_\om)$ of  bond percolation   on 
$G_\om$, and showed that $\pc(G_\om)<1$ for all $\om$. Such strict inequality
has been extended to graphs with super-linear growth by Duminil-Copin et al., \cite[Thm 1.3]{DGRSY}.
(The reader is referred to \cite{G99} for an introduction to percolation theory.) 
 The upper bounds of
Theorem \ref{thm:Grig2} provide lower bounds for the critical values $\pc(G_\om)$ via the well-known inequality
$\pc(H)\ge 1/\mu(H)$ for any connected, infinite, transitive graph $H$ (see the proof of \cite[eq.\ 1.13]{G99}).
\end{remark}

\section{Schreier graphs of Grigorchuk groups}\label{sec:abc}

The principal purpose of this section is to identify the labelled orbital Schreier graph 
$\sS_\om$ of the Grigorchuk 
group $G_\om$, and to deduce the asymptotic proportion of occurrences of a given label pair
of two parallel edges. 
We begin with a short review of Schreier graphs.

A \emph{ray} of $T$ is a SAW on $T$ starting at $\es$. The collection of all infinite rays is called
the \emph{boundary} of $T$ and denoted $\pd T$. Since each $\g\in\Aut(T)$
preserves the root $\es$, the orbit of any $v \in T$ is a subset of the generation of $T$ containing $v$.
Since $\g\in\Aut(G)$ preserves adjacency, $\g$ maps $\pd T$ into $\pd T$.
An element of $\pd T$ may be expressed as an infinite word with alphabet $\{0,1\}$,
for example, $0^210^31^\oo$. 

The \emph{orbit} $\Ga \rho$ of $\rho \in \pd T$  gives rise to a labelled graph, 
called the \emph{orbital Schreier graph} of $\rho$, and denoted here by $S(\rho)$.
The vertex-set of $S(\rho)$ is $\Ga\rho$. For $\rho_1,\rho_2\in \Ga\rho$,
$S(\rho)$ has an edge between $\rho_1$ and $\rho_2$ if and only if $\rho_2=x\rho_1$ for some $x \in \{a,b,c,d\}$; 
we label this edge with the generator $x$ and call it an $x$-edge. 
Such orbital Schreier graphs have been studied in \cite{Bar17,Grig-S,GLN,VY} and the references therein.

Let $1^\oo$ denote the rightmost infinite ray of $T$, 
with orbital Schreier graph $\sS:=S(1^\oo)$ illustrated
in Figure \ref{fig:one-end}.
It is standard (see, for example, \cite[Thm 7.3]{Grig-S} and \cite[p.\ 29]{VY})
that, if $\rho \in \Ga 1^\oo$, $S(\rho)$ is graph-isomorphic to 
the \emph{singly infinite} graph $\sS$
(the edge-labels depend on the choice of $\rho$). 
If $\rho \notin \Ga 1^\oo$, $S(\rho)$ is graph-isomorphic to a certain \emph{doubly infinite} chain which does not feature in this proof. For $\om\in\Om$, 
the orbital Schreier graph $\sS_\om$ is defined similarly relative to the generator-set
$\{a,b_\om,c_\om, d_\om\}$, and is illustrated in Figure \ref{fig:som}.

\begin{figure}[htbp]
\centerline{\includegraphics*[width=0.8\hsize]{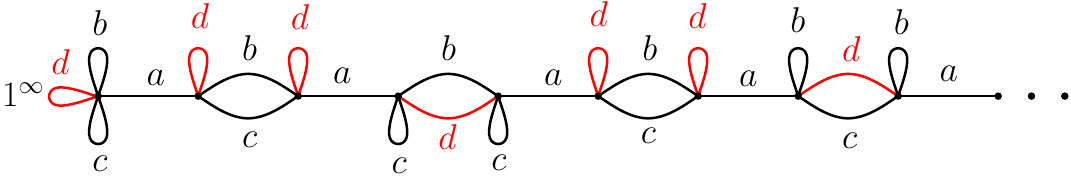}}
   \caption{The orbital Schreier graph $\sS$ of the ray $1^\oo$ for the first Grigorchuk group $\Ga$.
     The Schreier graph of $1^\oo$ associated with the smaller generator-set
$\{a,b,c, d\}\setm \{z\}$ is obtained by deleting all edges labelled $z$ (illustrated here with $z=d$ by
     drawing the edges and loops labelled $d$ in red).}
      \label{fig:one-end}
\end{figure}

\begin{figure}[htbp]
\centerline{\includegraphics*[width=0.8\hsize]{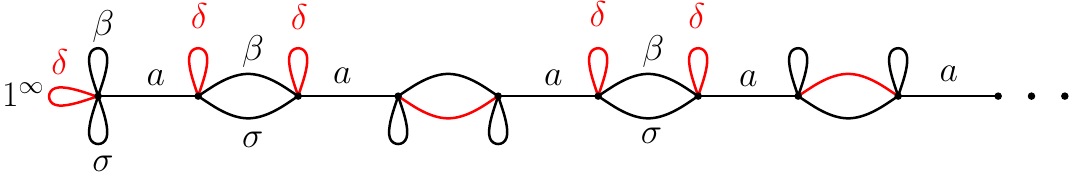}}
   \caption{The orbital Schreier graph $\sS_\om$ of the ray $1^\oo$ for the Grigorchuk group $\Ga_\om$
   with generator-set $\{a,\beta,\sigma,\de\}$.
     The automorphisms $\beta$, $\sigma$ are the two generators 
     with first entry $\Pi$, and $\delta$ is the remaining generator. Between consecutive pairs of $a$, there are two loops and
     two non-loops. Between the $k$th and $(k+1)$th appearances of $a$ (for $k$ odd), these edges are as labelled,
     whereas when $k$ is even these edge-labels depend on $\om$ and the edges' positions in $\sS_\om$.
     Deletion of the red edges gives rise to the Schreier graph $\sS_\om(\neg \de)$ associated with the generator-set
     $\{b_\om,c_\om,d_\om\}\setm\{\de\}$.}
      \label{fig:som}
\end{figure}

As remarked above, the first Grigorchuk group $\Ga$ may be identified as $\Ga_\om$ 
with $\om=(012)^{\oo}$.
Its orbital Schreier graph $\sS$ (see Figure \ref{fig:one-end}) is a map of the action of the
generator-set of $\Ga$ on the ray $1^\oo$.
Its leftmost vertex is $1^\oo$, and the progressive application of generators results in the sequence of 
rays given on the left side of Table \ref{fig:rays}.

\begin{table}[t]
\caption{The rays of $T$ corresponding to the first 25 vertices of the Schreier graphs $\sS_\om$
with (\emph{left}) $\om=(012)^\oo$ and (\emph{right}) $\om=(02112000\cdots)$.
When applied at any given vertex $v_k$ with $k\ne 1$, one generator creates a loop, and the other two 
lead to the next vertex $v_{k+1}$. The labels in the two columns differ from each other in the blue rows,
that is, at each $v_{4m}$ for $m\ge 1$ (they also differ at each $v_{4m+1}$, but that will
not concern us).}
\label{fig:rays}
\centerline{\begin{tabular}{|c|l|c|c|}
\hline
\multicolumn{4}{|c|}{First Grigorchuk group, $\om=(012)^\NN$} \\
\hline\hline
vertex $v_k$ & ray $\rho_k$ & loop & next labels\\
\hline\hline
1 & $\mathtt{1^\infty}$ & $b,c,d$ & $a$ \\
\hline
2 & $\mathtt{01^\infty}$ & $d$ & $b,c$ \\
\hline
3 & $\mathtt{001^\infty}$ & $d$ & $a$ \\
\hline
4 & $\mathtt{101^\infty}$ & $\ccl c$ &\ccl $b,d$ \\
\hline
5 & $\mathtt{1001^\infty}$ & $c$ & $a$ \\
\hline
6 & $\mathtt{0001^\infty}$ & $d$ & $b,c$ \\
\hline
7 & $\mathtt{0101^\infty}$ & $d$ & $a$ \\
\hline
8 & $\mathtt{1101^\infty}$ & \ccl $b$ &\ccl $c,d$ \\
\hline
9 & $\mathtt{11001^\infty}$ & $b$ & $a$ \\
\hline
10 & $\mathtt{01001^\infty}$ & $d$ & $b,c$ \\
\hline
11 & $\mathtt{00001^\infty}$ & $d$ & $a$ \\
\hline
12 & $\mathtt{10001^\infty}$ &\ccl $c$ &\ccl $b,d$ \\
\hline
13 & $\mathtt{10101^\infty}$ & $c$ & $a$ \\
\hline
14 & $\mathtt{00101^\infty}$ & $d$ & $b,c$ \\
\hline
15 & $\mathtt{01101^\infty}$ & $d$ & $a$ \\
\hline
16 & $\mathtt{11101^\infty}$ &\ccl $d$ &\ccl $b,c$ \\
\hline
17 & $\mathtt{111001^\infty}$ & $d$ & $a$ \\
\hline
18 & $\mathtt{011001^\infty}$ & $d$ & $b,c$ \\
\hline
19 & $\mathtt{001001^\infty}$ & $d$ & $a$ \\
\hline
20 & $\mathtt{101001^\infty}$ &\ccl $c$ &\ccl $b,d$ \\
\hline
21 & $\mathtt{100001^\infty}$ & $c$ & $a$ \\
\hline
22 & $\mathtt{000001^\infty}$ & $d$ & $b,c$ \\
\hline
23 & $\mathtt{010001^\infty}$ & $d$ & $a$ \\
\hline
24 & $\mathtt{110001^\infty}$ &\ccl $b$ &\ccl $c,d$ \\
\hline
25 & $\mathtt{110101^\infty}$ & $b$ & $a$ \\
\hline
\end{tabular}
\quad

\begin{tabular}{|c|c|}
\hline
\multicolumn{2}{|c|}{$\om=(02112000\cdots)$} \\
\hline \hline
loop & next labels\\
\hline\hline
 $b,c,d$ & $a$ \\
\hline
$d$ & $b,c$ \\
\hline
 $d$ & $a$ \\
\hline
\ccl $b$ &\ccl $c,d$ \\
\hline
 $b$ & $a$ \\
\hline
 $d$ & $b,c$ \\
\hline
 $d$ & $a$ \\
\hline
\ccl $c$ &\ccl $b,d$ \\
\hline
 $c$ & $a$ \\
\hline
 $d$ & $b,c$ \\
\hline
 $d$ & $a$ \\
\hline
\ccl $b$ &\ccl $c,d$ \\
\hline
 $b$ & $a$ \\
\hline
 $d$ & $b,c$ \\
\hline
 $d$ & $a$ \\
\hline
\ccl $c$ &\ccl $b,d$ \\
\hline
  $c$ & $a$ \\
\hline
 $d$ & $b,c$ \\
\hline
 $d$ & $a$ \\
\hline
\ccl $b$ &\ccl $c,d$ \\
\hline
$b$ & $a$ \\
\hline
 $d$ & $b,c$ \\
\hline
 $d$ & $a$ \\
\hline
\ccl $c$ &\ccl $b,d$ \\
\hline
 $c$ & $a$ \\
\hline
\end{tabular}}
\end{table}

Let $\om\in \Om$. Let $\de=\de(\om)\in\{b_\om,c_\om,d_\om\}$ be 
the unique element such that $\de_0=I$, and let $\beta$, $\sigma$ denote the other two generators.
According to \eqref{eq:012}, each $k\in\{0,1,2\}$ corresponds  to a $3$-vector $\ol k$,
and we let $\la_k$ be the generator whose element in $\ol k$ is $I$, and write
$\{\beta,\sigma,\de\}\setm \{\la_k\}=\{\mu_k, \tau_k\}$. 
Thus, 
\begin{equation}\label{eq:lambda}
\la_0=\delta, \q \{\la_1,\la_2\} =\{\beta,\sigma\}.
\end{equation}

The Schreier graph $\sS_\om$, illustrated in Figure \ref{fig:som}, is graph-isomorphic to 
$\sS$ but is labelled differently. The corresponding sequence of rays is identical to that of $\sS$, that is, as 
given in Table \ref{fig:rays}. The labels of $\sS_\om$ are as in Figure \ref{fig:som}, where two parallel unlabelled edges have distinct labels from $\{\beta,\sigma,\de\}$ and the incident loops have the
third label. The missing labels are assigned as follows.

Let $v_1,v_2,\dots$ be the vertices of $\sS_\om$ from left to right. Let $m\ge 1$ and consider $v_{4m}$
and the corresponding ray $\rho_{4m}$ of $T$ (see Table \ref{fig:rays}). 
The first letter of $\rho_{4m}$ is $1$, and 
there follows a sequence of $1$s followed by $0$; thus, we write $\rho_{4m} = 1^a0x$ where 
$a=a_m \ge 1$ and $x = x_m \in\{0,1\}^\oo$.
On following the prescription of Section \ref{sec:Grig}, we see that the loop 
of $\sS_\om$ at $v_{4m}$ is labelled
$\la_{\om_a}$, and the two rightward edges at $v_{4m}$ are labelled $\mu_{\om_a}$
and $\tau_{\om_a}$, \resp. The above construction may be described in diagrammatic form as
\begin{equation}\label{eq:X}
v_{4m} \larrow \rho_{4m} \larrow 1^a0x 
\larrow \om_a \larrow \la_{\om_a}.
\end{equation}

For given $\ell\in\{\beta,\sigma,\delta\}$, how frequently does $\ell$ appear as the label of the loop at some
 $v_{4m}$? The answer is implied by the next theorem. 
The indicator function of $A$ is written $1(A)$.  The set of finite words $w$ in the alphabet $\{0,1\}$ is denoted
$\{0,1\}^*$; the length of $w\in\{0,1\}^*$ is written $|w|$, and $\es$ 
denotes the empty word.
The word $w\in\{0,1\}^*$ is called \emph{periodic with period $p$}
if there exists $i=i(w)$ such that $\rho_{4m}$ begins with the sequence $1w$ if and only if $m=i+rp$
for some $r\ge 1$. 

\begin{theorem}\label{thm:label}
Each word $w\in\{0,1\}^*$ is periodic with period $2^{|w|}$.
In particular, the limit
\begin{equation*}
\pi_w:= \lim_{n\to\oo} \frac 1n \sum_{m=1}^n 1(\text{\rm$\rho_{4m}$ begins $1w$})
\end{equation*}
exists and satisfies
\begin{equation}\label{eq:new31}
\pi_w =  \frac1{2^{|w|}}.
\end{equation}
\end{theorem}

Theorem \ref{thm:label} may be expressed as saying 
that the empirical distribution of the rays $(\rho_{4m})$ converges 
weakly to the uniform distribution on the part of the boundary $\pd T$ that begins $1$. 
By consideration of the actions of the generators
of $\Ga_\om$ on the rays, this weak limit extends to the full sequence $(\rho_k)$
with the limit  being uniform on $\pd T$. 

\begin{corollary}\label{cor:1}
Let $\ell\in\{\beta,\sigma,\delta\}$. The limit
\begin{equation*}
\pi_\ell:= \lim_{n\to\oo} \frac 1n \sum_{m=1}^n 1(\text{\rm the loop at $v_{4m}$ is labelled $\ell$})
\end{equation*}
exists and satisfies
\begin{equation*}
\pi_\ell =  \sum_{a=1}^\oo  \frac1{2^{a}} 1(\la_{\om_a}=\ell) .
\end{equation*}
\end{corollary}

If, for example,  $\om$ is a vector of independent random variables, each being uniformly distributed
on $\{0,1,2\}$, the proportion of appearances of the loop-label $\ell$ may be written
as $Z=\sum_{a=1}^\oo X_a/2^{a}$ where the $X_a$ are independent with the Bernoulli distribution with parameter $\frac13$.

\begin{remark}\label{remL6}
It is likely that the above and perhaps more may be (or have been)
obtained using the theory of symbolic dynamics developed in \cite{AB, GNSL,GNS,VY} and elsewhere, but we prefer
to proceed directly here.
\end{remark}

\begin{proof}[Proof of Theorem \ref{thm:label}]
The mapping from the vertex-set of $\sS_\om$ into
the set $\pd T$ of infinite rays is an injection (see, for example, \cite[Thm 7.3]{Grig-S}, \cite[p.\ 29]{VY})
and is illustrated in Table \ref{fig:rays}. Note that this mapping is
independent of the choice of $\om$. The initial ray (that is, the image of the root of $\sS_\om$) is the ray $\rho_1=1^\oo$. We shall consider the evolution of the sequence $(\rho_k)$ of rays as $k$ grows. It
is convenient to concentrate on suffices $k$ that are multiples of $4$, and we call the evolution
from any $\rho_{4m}$ to $\rho_{4(m+1)}$ a \emph{stage} of the process.

It is immediate that the first letter of every $\rho_{4m}$ is $1$. Next we perform an illustration
of the required argument.

Suppose $m$ is such that
$v_{4m}= 11x$ for some $x\in \{0,1\}^\oo$, $x\ne 1^\oo$. The four moves that constitute the next stage
are as follows.
\begin{numlist}
\item The first letter following the first $0$ is flipped, resulting in $11x'$ for appropriate $x'$.
\item The first letter of the ensuing word is flipped, giving $01x'$.
\item The second letter of the ensuing word is flipped, giving $00x'$.
\item The first letter of the ensuing word is flipped, giving $10x'$.
\end{numlist}
This sequence results in the word  $\rho_{4m+4}=10x'$. We express this by writing
\begin{equation}\label{eq:11}
11 \larrow 10.
\end{equation}
On applying the same argument to the word $10y$ with $y\in\{0,1\}^\oo$, we find that
$10 \larrow 11$, and thus we may write
\begin{equation}\label{eq:12}
11 \larrow 10 \larrow 11.
\end{equation}
By \eqref{eq:11}--\eqref{eq:12}, the words $0$ and $1$ are periodic with period $2$.
Hence, $\pi_0=\pi_1=\frac12$ in agreement with the claim of the theorem. 

The corresponding relations with three letters are
\begin{equation}\label{eq:12a}
111\larrow 101 \larrow 110 \larrow 100\larrow 111,
\end{equation}
so that the limits $\pi_{11}$, $\pi_{01}$, $\pi_{10}$, $\pi_{00}$ exist and each equals $\frac14$.

It is instructive to combine \eqref{eq:12} and \eqref{eq:12a} in a vertical table. Table \ref{tab:v1}
illustrates the nested structure of the relations \eqref{eq:12}--\eqref{eq:12a}. Some comments follow.
\begin{letlist}
\item Let $a\ge 1$, and suppose we have tabulated the evolution of the initial $(a+1)$-subsequences 
of the rays, starting at $1^{a+1}$, and continuing until the first reappearance of $1^{a+1}$.
\item Let $A_a$ be the statement  that, for every $w\ne 1^a$ with $|w|=a$, the word $1w$ appears in
this tabulation exactly once. (The word $1\cdot 1^a$ appears twice, once at each end.)
Table \ref{tab:v1} includes such tabulations for $a=1,2$, and the assumption is satisfied for these values..
\item
\emph{Assume $A_a$ holds.}
Take two copies of this tabulation, the second of which is appended beneath the first. We add 
the appropriate $(a+2)$th column, noting that it starts with its letter $1$, and that that it 
changes its value each time it passes the word $1^a0$ in the left $a+1$ columns, and not elsewhere. 
\item The word $1^a0$ occurs exactly once in each of the two tabulations, and at the same position.
\item It follows by $A_a$ that $A_{a+1}$ holds. Since $A_1$ holds, every $A_a$ holds by induction. 
\end{letlist}

\begin{table}[h]
\caption{\emph{Left}: Two successive copies of the display \eqref{eq:12} are displayed vertically
on the left side, followed in the third column by the extensions given in \eqref{eq:12a}. 
As one descends the third column, there are exactly two changes in value, at the cells coloured blue,
each corresponding to the previous word 10 (coloured pink) in the two left columns. \emph{Right}: 
Two copies of \eqref{eq:12a} displayed vertically. The two changes in the fourth column correspond to the occurrences of 110 in the previous row. }\label{tab:v1}
\centerline{\begin{tabular}{cc|c}
1 & 1 & 1\\
\hline\hline
\cclr 1 &\cclr  0 & 1 \\
\hline
1 & 1 & \ccl 0 \\
\hline\hline
\cclr 1 &\cclr 0 & 0 \\
\hline
1 & 1 & \ccl 1 \\
\hline
\end{tabular}
\qquad\qquad
\begin{tabular}{ccc|c}
1 & 1 & 1 & 1\\
\hline\hline
1 & 0 & 1 & 1\\
\hline
\cclr1 &\cclr 1 &\cclr 0 & 1\\
\hline
1 & 0 & 0 &\ccl 0\\
\hline
1 & 1 & 1 & 0\\
\hline\hline
1 & 0 & 1 & 0\\
\hline
\cclr 1 &\cclr 1 &\cclr 0 & 0\\
\hline
1 & 0 & 0 & \ccl 1\\
\hline
1 & 1 & 1 & 1\\
\hline
\end{tabular}}
\end{table}

In conclusion, for every $a\ge 1$, the evolution beginning $1^{a+1}$, until its first reappearance,
includes every $1w$ with $|w|=a$, $w\ne 1^a$, exactly once. That is to say, the evolution
\begin{equation}\label{eq:13}
1^{a+1} \larrow 101^{a-1} \larrow 1101^{a-2}\larrow10^31^{a-3}\larrow  \cdots \larrow 1001^{a-2} \larrow 1^{a+1} ,
\end{equation}
includes a unique copy of any word of the form $1w$ with $|w|=a$, $w\ne 1^a$.
Therefore, such $w$ has period $2^{|a|}$, whence the limit $\pi_w$ exists for all $w$ and satisfies \eqref{eq:new31}.
\end{proof}

\begin{proof}[Proof of Corollary \ref{cor:1}]
Let $a\ge 1$, $\ell\in\{\beta,\sigma,\delta\}$, and let $h_m(a,\ell)$ be the indicator function that 
$a_m=a$ and in addition $\la_{\om_a}=\ell$. (Recall \eqref{eq:lambda} and the notation
leading to \eqref{eq:X}.)
Then,
\begin{align*}
\pi_\ell = \lim_{n\to\oo} \frac 1n \sum_{m=1}^n \sum_{a=1}^\oo h_m(a,\ell),
\end{align*}
and the claim holds by Theorem \ref{thm:label} on interchanging the summations.
\end{proof}

\section{Proof of Theorem \ref{thm:Grig2}}\label{sec:pfmain}

The proof is basically the same whatever the choice of $\om$, and therefore the reader loses little by
restricting attention to the case
$\om=(012)^\oo$ and $b_\om=b=\beta$, $c_\om=c=\sigma$, $d_\om=d=\delta$.
For completeness we shall include the details even in places where  they resemble the proof of \cite[Thm 3.1]{GL20}. 

Throughout this proof we assume $\om\in\Om$, and we let $\beta$, $\sigma$, $\de$ be given as in the statement
of the theorem, that is, $\de=\de(\om)\in\{b_\om,c_\om,d_\om\}$ is such that $\de_0=I$, and 
$\beta$, $\sigma$ are the two remaining elements of $\{b_\om,c_\om,d_\om\}$. 

\subsection{Case (a)}\label{ssec:G}

Let $\sW_0$ be the set of finite, labelled walks on the
Schreier graph $\sS_\om$ starting at the root $1^\oo$ that, at each step, either move one
step rightwards or pass around a loop (no loop may be traversed more than once).
Members of $\sW_0$ may be considered as finite words 
in the alphabet $\{a,\beta,\sigma,\de\}$ without consecutive repetitions. 
Walks in $\sW_0$  generally contain loops and are therefore not self-avoiding on $\sS_\om$, but we shall
see next that they give rise to a certain set $\ol{\sW_0}$
of self-avoiding walks on $G_\om$ starting at its root $\id$.

Let $\sW_0(\oo)$ be the set of \emph{infinite} words constructed from $\sS_\om$ subject to the rules used in defining $\sW_0$,
and let $\ol {\sW_0}(\oo)$ be the set of walks on $G_\om$ obtained as lifts of elements of $\sW_0(\oo)$.

Recall the set $\sK$ of kites.
Let $w=w_0w_1\cdots \in \sW_0(\oo)$ and $\kappa\in\sK$. Write
$w^n=w_0w_1\cdots w_n$, a walk/word on $\sS_\om$ that lifts to the walk $\ol{w^n}$ on $G_\om$.
We say that $\ol w$ \emph{enters} $\kappa$
at epoch $n$ if either $n=0$ and $\ol{w^0}\in \kappa$, or $n\ge 1$, $\ol{w^n}\in\kappa$, $\ol{w^{n-1}}\notin\kappa$.

\begin{lemma}\label{lem:0}\mbox{\hfil}
\begin{letlist}
\item
Every $w\in\sW_0(\oo)$ lifts to an infinite SAW $\ol w$ in $G_\om$ from the root $\id$.

\item
For $w\in \sW_0(\oo)$ and $\kappa\in\sK$, the infinite SAW $\ol w = (w_1,w_2,\dots)$ does not enter $\kappa$ 
more than once.

\item
Every $\ol w\in \ol{\sW_0}$ is the initial sub-walk of an infinite SAW 
lying in $\ol{\sW_0}(\oo)$.
\end{letlist}
\end{lemma}

\begin{proof}
(a) Let $w\in\sW_0(\oo)$ and suppose $\ol w$ is
not a SAW. Then $w$ contains some shortest subword  $s$ of length $3$ or more
satisfying $s=\id$. On considering the action of $\Ga_\om$ on $\sS_\om$
(see Figure \ref{fig:som}), we deduce
that $\sS_\om$ contains a cycle of length $3$ or more. By inspection of $\sS_\om$, this is 
seen to be a contradiction.

(b) Suppose $\ol w\in \ol{\sW_0}(\oo)$ enters some $\kappa\in\sK$ twice (or more often).
In the order of traversal of  $\ol w$, write $v_1$ for the first 
departure point from $\kappa$, and $v_2$ for the first subsequent 
arrival point in $\kappa$.
Then $v_1$, $v_2$ are at least distance $4$ apart in $\ol w$.  
Since $v_1,v_2\in \kappa$, there exists a generator $z$ ($\ne a$) such that $v_2=v_1z$, whence $\sS_\omega$ contains 
some cycle of length $4$ or more, a contradiction. 

(c) This holds since every word in $\sW_0$ may be extended to an infinite word in $\sW_0(\oo)$.
\end{proof}

Let $R=\{r_1, r_2,\dots\}$ be the set of right-hand 
endpoints of the $a$-edges of $\sS_\om$
labelled in the order they are encountered
when moving to the right from $1^\oo$ (in the sense of Figure \ref{fig:som}). An ordered pair of elements $r_i,r_j\in R$ is called
\emph{consecutive} if $|i-j|=1$. Let $\sW$ be the subset of $\sW_0$
containing words that end in $a$. 
As above, $\sW$ lifts to a set $\ol\sW$ of SAWs on $G_\om$.

We think of the set $R$ as being points of renewal of walks in $\sW$.
More specifically, each $w\in \sW$ can be broken into sections  (called \emph{units})
beginning and ending (\resp) with a  consecutive pair $z, z'\in R$,
and each  unit $\si$ has the form of Figure \ref{fig:unit1}.

\begin{figure}[htbp]
\centerline{\includegraphics*[width=0.3\hsize]{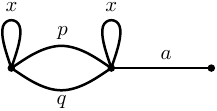}}
   \caption{A unit of the Schreier graph $\sS_\om$ with general labels. 
   The counts (according to length) of SAWs from the left end to the right end of the unit
   have generating function $A(\xi)=2\xi^2+4\xi^3+2\xi^4$.}
      \label{fig:unit1}
\end{figure}

The generating function 
\begin{equation}\label{eq:gf1}
Z(\xi)=\sum_{w\in\ol{\sW}} \xi^{|w|}
\end{equation}
 of $\ol{\sW}$ may be expressed in the form
\begin{equation}\label{eq:gf2}
Z(\xi) = A_0(\xi)\sum_{n=0}^\oo A(\xi)^n,
\end{equation}
where $A_0(\xi)$ is the generating function of walks from $1^\oo$ to 
$r_1$, and $A$ is given in the caption of Figure \ref{fig:unit1}.
Since $Z(\xi)\le Z_{G_\om}(\xi)$ (for $\xi\ge 0$) and $Z_{G_\om}$ has radius of convergence $1/\mu(G_\om)$
(recall \eqref{eq:ZG}), the radius of convergence of $Z$ is at least $1/\mu(G_\om)$, 
whence $\mu(G_\om)\ge \zeta$ where $1/\zeta$ 
is the positive root of the equation $A(\xi)=1$. This provides a lower bound for $\mu(G_\om)$ that may be 
improved somewhat by the following variant of the argument used in \cite[Sect.\ 8]{GL20}.

By Lemma \ref{lem:0}, each SAW $\ol w\in \ol{\sW_0}$ visits no kite $\kappa$ more than once,
and its route through such $\kappa$ can be any of the eight words implied by the labelling of the associated unit of
$\sS_\om$. Using Figure \ref{fig:unit1} for illustration, these words are  $pa$, $qa$, $xpa$, $xqa$,
$pxa$, $qxa$, $xpxa$, $xqxa$. To these we may add the two words $pqpa$, $qpqa$; see Figure \ref{fig:G2}
for an illustration of the two corresponding SAWs 
within $\kappa$. The above function $A(\xi)$ is now replaced by 
$\what A(\xi)=
2 \xi^2+4\xi^3+4\xi^4$, whence $\mu(G_\om)$ is bounded below by the reciprocal of the positive root of
$\what A(\xi)=1$.

Turning to  the upper bound on $\mu(G_\om)$, one may use the fact that 
$G_\om$ has degree $4$ and girth $3$, and apply  \cite[Thm 7.3]{GL20}. An improved upper bound
is obtained by considering the Cayley graph $G'$ of the group $\Ga'$
obtained from $\Ga_\om$ by removing all relators except \eqref{eq:commute}.
This $G'$ is isomorphic to the free product  graph $K_2 * K_4$ which, by \cite[Thm 3.3]{Gilch},
has connective constant $\mu'$ given as the positive root of $\xi^4-3\xi^2-6\xi-6=0$.
(Here, $K_m$ denotes the complete graph on $m$ vertices.)
By \cite[Cor.\ 4.1]{GL14}, we have the strict inequality $\mu(G_\om) < \mu'$.

A slightly improved upper bound holds by an approach proposed by Rudi Mrazovi\'c
\cite{RM}, as follows.
Let $\de=\de(\om)\in\{b_\om,c_\om,d_\om\}$ be such that $\de_0=I$, and let $\be$ and $\si$ be the generators with $\be_0=\si_0=\Pi$.
Note that $(a\delta)^4=\id$.  Let $H$ be the graph
obtained from $G_\om$ be contracting every kite to a point (recall Figure \ref{fig:G2}). 
Since $(a\de)^4=\id$, every vertex of $H$ lies in a cycle of length $4$; such a cycle is called  $\de$-cycle.

We may think of a walk $w$ on $G_\om$ as a walk $w_H$ on $H$ 
combined with its sub-walks within the kites encountered. As $w$ passes through a kite
$\kappa$,  we call $\kappa$ a \emph{$\de$-kite} if the entry vertex $e$ (\resp, departure vertex $f$)
of $w$ at $\kappa$ satisfy $f=e\de$. Note that every SAW $w$ of $G_\om$ has the property that it has no sequence of $4$ consecutive $\de$-kites (since otherwise $w_H$ would contain a $\de$-cycle).

Let $\sV=\bigcup_n\sV_n$ where $\sV_n$ is the set of $n$-step 
walks $v=(v_0,v_1,\dots,v_n)$ on $G_\om$ such that 
(i) $v_0=\id$, $v_1\in \{\be,\si,\de\}$, (ii) $v$ is a SAW within each kite encountered,
(iii) $v_n$ is the entry vertex of some kite,       
and (iv) $v$ has no subsequence of $4$ consecutive $\de$-kites. 
We will calculate the radius of convergence $\rho$ of the generating function 
\begin{equation*}
Z_V(\xi)=\sum_{v\in\sV} \xi^{|v|},
\end{equation*}
and deduce that $\mu\le 1/\rho$.

The generating function of SAWs within a $\de$-kite is
$\xi+2\xi^2+2\xi^3$, so that, on taking into account the following $a$-step, this contributes a term
$D(\xi):=\xi^2+2\xi^3+2\xi^4$ to $Z_V$.  The other two non-trivial displacements 
within a given kite each contribute $D$ also, and their combined generating function is $B(\xi)=2D(\xi)$.
By breaking a walk $v\in\sV$ into sub-walks each using no more than $3$ consecutive $\de$-kites
followed by a kite which is not a $\de$-kite,
we find that 
\begin{equation*}
Z_V(\xi) = \sum_{n=1}^\oo C(\xi)^n,
\end{equation*} 
where 
$C(\xi):=(1+D(\xi)+D(\xi)^2+D(\xi)^3)B(\xi)$. 
Therefore, $\mu$ is bounded above by the reciprocal of the positive root of $C(\xi)=1$.

\subsection{Cases (b, c)}\label{ssec:caseb}

We follow the proof of part (a) wherever possible. Let $z\in \{\beta,\sigma,\de\}$.
The orbital Schreier graph $\sS_\om(\neg z)$ of $\Ga_\om$
is obtained from $\sS_\om$ by deleting all elements labelled $z$; see Figure \ref{fig:Sch-db}
for an illustration of the cases $z=\beta,\de$.  

Let $\sW_0(\neg z)$ be the set of finite, labelled walks on $\sS_\om(\neg z)$ 
starting at the root $1^\oo$ that, at each step, either move one
step rightwards or pass around a loop.
We view $\sW_0$ as a set of finite words 
in the alphabet $\{a,\beta,\sigma,\de\}\setm \{z\}$ without consecutive repetitions. 

\begin{figure}[t]
\centerline{\raisebox{20pt}{$\sS_\om(\neg \de)$:}\hfil\includegraphics*[width=0.74\hsize]{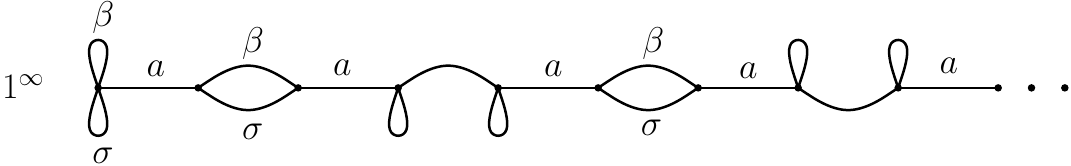}}
\centerline{\raisebox{20pt}{$\sS_\om(\neg \beta)$:}\hfil\includegraphics*[width=0.74\hsize]{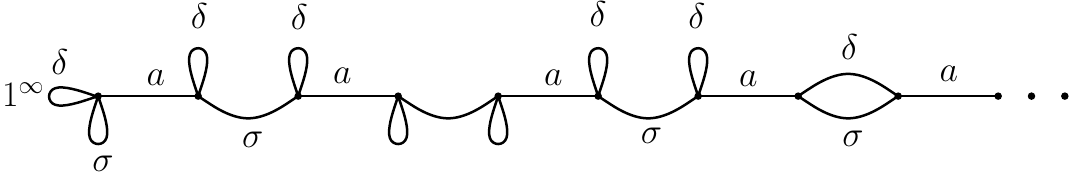}}
   \caption{The one-ended orbital Schreier graphs $\sS_\om(\neg \de)$ and $\sS_\om(\neg \beta)$
   of the ray $1^\oo$ for the Grigorchuk group $\Ga_\om$ with respective generator sets  $\{a,\beta,\sigma\}$
   and $\{a,\sigma,\de\}$. The Schreier graph $\sS_\om(\neg \sigma)$ is similar to $\sS_\om(\neg \beta)$.}
      \label{fig:Sch-db}
\end{figure}

As in part (a), each  $w \in \sW_0(\neg z)$ lifts to a distinct walk $\ol w$ on $G_\om(\neg z)$, and 
furthermore Lemma \ref{lem:0}
holds. We now examine the units of the Schreier graph $\sS_\om(\neg z)$, beginning with the case  $z=\de$.

\subsubsection{The case $z=\de$}\label{ssec:d}

Let $\sW$ be the subset of $\sW_0(\neg \de)$
containing words that end in $a$. 
As above, $\sW$ lifts to a set $\ol\sW$ of SAWs on $G$.
It turns out that $\sW$ is not sufficiently large to obtain an adequate lower bound for $\mu(G_\om(\neg \de))$,
and therefore we shall augment $\sW$ to a larger set $\sW'$ of words. This will be done
by considering the individual units of $\sS_\om(\neg \de)$. The unit
between $r_i$ and $r_{i+1}$
is called \emph{odd} (\resp, \emph{even})
if $i$ is odd (\resp, even). It may be seen that every odd unit 
is a copy of the unit $u_1$ of Figure \ref{fig:units}, and every even unit either is a copy of $u_1$ or 
contains a copy of $u_2$
(of the same figure).

\begin{figure}[htbp]
\centerline{\includegraphics*[width=0.7\hsize]{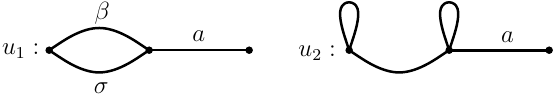}}
   \caption{Two units contributing to $\sS_\om(\neg \de)$, denoted $u_1$ (left) and $u_2$ (right).
   The omitted labels in $u_2$ depend on the position of the unit in the Schreier graph.}
      \label{fig:units}
\end{figure}

Let $\xi\ge 0$. As in \eqref{eq:gf1}--\eqref{eq:gf2}, the generating function 
\begin{equation}\label{eq:gf3}
Z(\xi)=\sum_{w\in\ol{\sW}} \xi^{| w|}
\end{equation}
 of $\ol{\sW}$ satisfies (for $\xi\ge 0$)
\begin{equation}\label{eq:gf4}
Z(\xi) \ge A_0(\xi)\sum_{n=0}^\oo \prod_{i=1}^n A_i(\xi),
\end{equation}
where $A_0(\xi)$ is the generating function of walks from $1^\oo$ to $r_1$, and
\begin{equation*}
A_i(\xi) = \begin{cases}
2\xi^2 &\text{when $i$ is odd},\\
\text{either $2\xi^2$ or $\xi^2+2\xi^3 +\xi^4$}&\text{when $i$ is even.}
 \end{cases}
\end{equation*}
This can be improved by augmenting the set of walks in copies of the unit $u_1$.

Walks from left to right in $u_1$ correspond to the set $L:=\{\beta a,\sigma a\}$ of words. 
As in the corresponding part of the proof of part (a), we augment this set to the set
$L'=\{\beta a,\sigma a,\beta \sigma\beta a,\sigma\beta\sigma a\}$. This has the effect of replacing 
$A_i(x)$, for odd $i$, by
$\what A(\xi):= 2\xi^2 + 2\xi^4$. Note that
\begin{equation}\label{eq:new1}
\what A(\xi) \ge \xi^2+2\xi^3 +\xi^4.
\end{equation}
As in the proof of part (a), the walks in $G_\om(\neg\de)$ thus obtained are indeed SAWs.

Returning to \eqref{eq:gf3}--\eqref{eq:gf4}, and grouping each odd unit with the following even unit, we deduce 
by \eqref{eq:new1} that 
\begin{equation*}
Z(\xi) \ge  A_0(\xi) \sum_{n=0}^\oo  H(\xi)^n
\end{equation*}
where 
\begin{equation*}
H(\xi)= \what A(\xi)(\xi^2+2\xi^3 +\xi^4) = 2\xi^4(1 + \xi^2)(1 +\xi)^2.
\end{equation*} 
Recalling \eqref{eq:ZG}, it follows that $\mu\ge\gamma\approx 1.635$ where $1/\gamma$ is the 
positive root of $H(\xi)=1$. The lower bound of the theorem is proved. 

Let $K_2 * \ZZ_4$ denote the free product (Fisher) graph of an edge and a $4$-cycle
(see the discussion around \cite[Thm 7.2]{GL20}).
The upper bound $\mu\le\mu(K_2*\ZZ_4) \approx 1.899$ may be obtained by applying
\cite[Thm 7.2]{GL20}, on noting that $G_\om(\neg \de)$ has degree $3$ and girth $4$. That this upper bound is a \emph{strict} inequality holds by the remark concerning the graph $K_2 * \ZZ_4$ in \cite{GL20},
and the strict-inequality theorem \cite[Cor.\ 4.1]{GL14}. This bound may be improved by using Mrazovi\'c's argument (see the proof of part (a)) with the given functions $B$ and $D$.

\subsubsection{Case (c)}\label{ssec:b}

Since the arguments for $z=\beta,\sigma$  are essentially the same, we may suppose $z=\beta$.
The orbital Schreier graph $\sS_\om(\neg \beta)$ is a line of copies of $u_2$ interspersed by a number of appearances of $u_1$, with the last labelled as  in Figure \ref{fig:units-b}.
The words represented in such a copy of $u_1$ are $\{\sigma a,\de a\}$. 
We augment this set to $\{\sigma a,\de a,\sigma\de\sigma a, \de\sigma\de a\}$,
as in the style of the last proof. As in \eqref{eq:gf4},
we deduce that
\begin{equation}\label{eq:new4}
Z(\xi) \ge A_0(\xi)\sum_{n=0}^\oo \prod_{i=1}^n A_i(\xi),
\end{equation}
where each $A_i$ (with $i\ge 1$) is either $A_1(\xi):= 2\xi^2(1+\xi^2)$ or 
$A_2(\xi) := \xi^2(1+\xi)^2$. It is easily checked that $A_1(1/\phi), A_2(1/\phi) \ge 1$, so that
$Z(1/\phi)=\oo$, whence, by \eqref{eq:ZG}, $\mu(G_\om(\neg\beta))\ge \phi$. 

Turning to the upper bound on $\mu$, we have that $\mu<\mu(K_2*\ZZ_4)$
as in the proof of part (b). This 
bound may be improved using Mrazovi\'c's argument (from
the proof of part (a)) to obtain $\mu\le \gamma''$ with $\gamma''$ as given.

\begin{figure}[t]
\centerline{\includegraphics*[width=0.25\hsize]{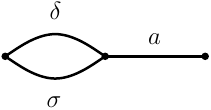}}
   \caption{The labels on every copy of $u_1$ in $\sS_\om(\neg \beta)$. 
   Note the missing loops of type $\beta$.}
      \label{fig:units-b}
\end{figure} 

Corollary \ref{cor:1} provides a quantification of the number of appearances of $u_1$,
and hence an improvement over the lower bound derived above.
Let $\pi$ be the asymptotic proportion of units $u_1$ amongst units at vertices $v_{4m}$ 
of $\sS_\om(\neg \beta)$ for $m\ge 1$. By Corollary \ref{cor:1}, 
\begin{equation}\label{eq:pos}
\pi = \sum_{a=1}^\oo  \frac1{2^{a}} 1(\la_{\om_a}=\beta),
\end{equation}
which is strictly positive whenever $\om$ has the property that there exists $a\ge 1$ such that $\la_{\om_a}=\beta$.
When this holds, we may proceed as in \eqref{eq:gf4} and \eqref{eq:new4} to find that 
$\mu(G_\om(\neg\beta))\ge \zeta$ where $\zeta$ is the reciprocal of the positive root of
the equation
\begin{equation}\label{eq:improve}
[A_1(\xi)]^\pi [A_1(\xi)]^{1-\pi}A_2(\xi)-1=0.
\end{equation} 
This root is strictly smaller than $1/\phi$.

\section*{Acknowledgements}
The author thanks Tatiana Nagnibeda for asking about connective constants of \emph{general} Grigorchuk graphs.
The proposals of Anton Malyshev concerning SAWs and Schreier graphs 
have been useful in both \cite{GL20} and the current work.  Rudi Mrazovi\'c has
very kindly commented on the work and made useful proposals for certain numerical improvements
reported in Theorem \ref{thm:Grig2}.
ChatGPT wrote a Python program at the request of the author.

\section*{Declaration}
The author has no conflicts of interest, and no data arose in this work.

\providecommand{\bysame}{\leavevmode\hbox to3em{\hrulefill}\thinspace}
\providecommand{\MR}{\relax\ifhmode\unskip\space\fi MR }
\providecommand{\MRhref}[2]{%
  \href{http://www.ams.org/mathscinet-getitem?mr=#1}{#2}
}
\providecommand{\href}[2]{#2}

\end{document}